\documentclass[11pt]{amsart}

\usepackage{amsmath,amssymb,latexsym,soul,cite,mathrsfs}

\usepackage{color,enumitem,graphicx}
\usepackage[colorlinks=true,urlcolor=blue,
citecolor=red,linkcolor=blue,linktocpage,pdfpagelabels,
bookmarksnumbered,bookmarksopen]{hyperref}
\usepackage[english]{babel}

\usepackage[left=2.9cm,right=2.9cm,top=2.8cm,bottom=2.8cm]{geometry}
\usepackage[hyperpageref]{backref}

\usepackage[colorinlistoftodos]{todonotes}
\makeatletter
\providecommand\@dotsep{5}
\def\listtodoname{List of Todos}
\def\listoftodos{\@starttoc{tdo}\listtodoname}
\makeatother

\numberwithin{equation}{section}

\newtheorem{theorem}{Theorem}[section]
\newtheorem{proposition}[theorem]{Proposition}
\newtheorem{lemma}[theorem]{Lemma}

\newcommand\R{\mathbb R}

\begin{document}

\title[ Existence of solution for a class of problem in whole .......]
{Existence of solution for a class of problem in whole $\mathbb{R}^N$ without the Ambrosetti-Rabinowitz condition }

\author{Claudianor O. Alves}
\author{Marco A. S. Souto }

\address[Claudianor O. Alves and Marco A.S. Souto]{\newline\indent Unidade Acad\^emica de Matem\'atica
\newline\indent 
Universidade Federal de Campina Grande,
\newline\indent
58429-970, Campina Grande - PB - Brazil}
\email{\href{mailto:coalves@mat.ufcg.edu.br, marco@mat.ufcg.edu.br}{coalves@mat.ufcg.edu.br, marco@mat.ufcg.edu.br}}


%

\pretolerance10000


\begin{abstract} In this paper we study the existence of solution for a class of elliptic problem in whole $\mathbb{R}^N$ without the well known Ambrosetti-Rabinowitz condition. Here, we do not assume any monotonicity condition  on $f(s)/s$ for $s>0$.  

\end{abstract}

\thanks{Claudianor Alves was partially supported by CNPq/Brazil Proc. 304804/2017-7 ; Marco A.S. Souto was partially
supported by CNPq/Brazil Proc. 306082/2017-9 }
\subjclass[2010]{Primary 35A15;  Secondary 35B38; 35J20 } 
\keywords{Variational methods; Critical points; Superlinear problems; Elliptic equations}

\maketitle

\section{Introduction}
 
In the last decades a lot of authors have studied the existence of solution for elliptic problems of the form
$$
\left\{
\begin{array}{l}
-\Delta{u}=f(x,u), \quad \mbox{in} \quad \Omega \\
u=0, \quad \mbox{on} \quad \partial \Omega,
\end{array}
\right.
\leqno{(P)}
$$
where $\Omega \subset \mathbb{R}^N$ is a smooth bounded domain and $f$ is a continuous function with subcritical growth satisfying some technical conditions. In general, the main conditions on $f$ are the following: \\
$$
 \displaystyle \lim_{s \to 0}\frac{f(s)}{s}=0. \leqno{(f_1)}
$$
\noindent There is $q \in (2,2^*)$, where $2^*=\frac{2N}{N-2}$, if $N \geq 3$ or $q \in (2,+\infty)$ if $N=1,2$ such that
$$
\displaystyle \limsup_{|s| \to +\infty }\frac{|f(s)|}{|s|^{q-1}}<+\infty. \leqno{(f_2)}
$$ 
\noindent (Ambrosetti-Rabinowitz condition) There are $\theta >2$ and $M>0$ such that
$$
0<\theta F(s) \leq f(s)s, \quad \quad \mbox{for} \quad |s| \geq M, \leqno{(f_3)}
$$
where $F(s)=\displaystyle \int_{0}^{s}f(t)\,dt$.

By using the Mountain Pass Theorem due to Ambrosetti and Rabinowitz \cite{AR}, it is possible to prove that the energy functional associated with $(P)$, given by
$$
I(u)=\frac{1}{2}\int_{\Omega}|\nabla u|^{2}\,dx-\int_{\Omega}F(u)\,dx, \quad \forall u \in H_{0}^{1}(\Omega),
$$
satisfies the well known $(PS)$ condition. Here, assumption $(f_3)$ permits to prove, of a very easy way, that all $(PS)$ sequences are bounded. However, in the last years, we have observed that in a lot papers, the authors have used a more weak condition that $(f_3)$, more precisely, instead of $(f_3)$ the following conditions are assumed:
$$
 \frac{F(s)}{s^2} \to +\infty \quad \mbox{as} \quad |s| \to +\infty. \leqno{(f_4)} \\
$$
and 
$$
 \mbox{The function} \quad \frac{f(s)}{|s|} \quad \mbox{ is an increasing function of} \quad s \in \mathbb{R} \setminus \{0\}. \leqno{(f_5)}
$$
The condition $(f_5)$ is sometimes  replaced by
$$
\mathcal{F}(s)=f(s)s-2F(s) \quad \mbox{is increasing for} \quad \mbox{for} \quad  s >0, \, \mbox{and decreasing for} \, s<0. \leqno{(f_5)'}
$$

 The literature is large for problems without Ambrosseti-Rabinowitz condition, we would like to cite the papers by Costa and Magalh\~aes \cite{Costa}, Jeanjean and Tanaka \cite{JeanjeanTanka}, Liu and Wang \cite{Liu-Wang}, Miyagaki and Souto \cite{Miyagaki-Souto}, Schechter and Zou \cite{SZ}, Struwe and Tarantello \cite{ST}, Wang and Wei \cite{Wang-Wei}, Zhou \cite{Zhou} and their references. In general, the main tool used in the above mentioned papers is Mountain Pass Theorem with Cerami's condition found in Bartolo, Benci and  Fortunato \cite{BBF}.  

Here, we would like point out that in the seminal paper \cite{Costa}, Costa and Magalh\~aes established the existence of solution for $(P)$ without Ambrosseti-Rabinowitz condition by supposing, among others conditions,  
$$
\displaystyle \liminf_{|s| \to +\infty }\frac{f(s)s-F(s)}{|s|^{\mu}} \geq a>0, \quad 	\leqno{(f_6)} 
$$
with $\mu > \frac{N}{2}(q-2)$ and $q$ as in $(f_2)$. Since that $\Omega$ is a bounded domain, the authors were able to show that $(f_2)$ and $(f_6)$ combine to give the boundedness of $(PS)$ sequences for the energy functional associated with $(P)$, which is a key point to prove  $(PS)$ condition.  In that paper, an interesting point is that the authors did not assume any monotonicity condition on $\frac{f(s)}{s}$ or $\mathcal{F}(s)=f(s)s-2F(s)$ for $s>0$.

The existence of solution for elliptic problem in whole $\mathbb{R}^N$ like  
$$
\left\{
\begin{array}{l}
-\Delta{u}+V(x)u=f(u), \quad \mbox{in} \quad \mathbb{R}^N, \\
u \in H^{1}(\mathbb{R}^N)
\end{array}
\right.
\leqno{(P_1)}
$$
without Ambrosseti-Rabinowitz condition also have been considered in some papers, see for example, Jeanjean \cite{Jeanjean}, Liu \cite{Liu1} and their references.  In \cite{Jeanjean}, Jeanjean has proved a very interesting Abstract Theorem that permits to work with a large class of problem without Ambrosetti-Rabinowitz condition in bounded or unbounded domain. In that paper the author assumes $V=1$ and that there is $D \in [1, +\infty)$ such that 
\begin{equation} \label{ZE}
Z(s) \leq DZ(t), \quad \mbox{for} \quad 0 \leq s \leq t,
\end{equation}
where $Z(s)=\frac{1}{2}f(s)s-F(s)$ for $s \geq 0$. In \cite{Liu1}, Liu used essentially the same arguments explored in \cite{Miyagaki-Souto}, \cite{SZ} and \cite{ST} by supposing that $V$ is $\mathbb{Z}^N$-periodic or 
$$
V(x)<V_\infty=\lim_{|x| \to +\infty}V(x).
$$ 
Related to the function $f$, Liu also assumed a condition like (\ref{ZE}).

Motivated by the above references, the present paper is concerned with the existence of solution for $(P_1)$ without Ambrosseti-Rabinowitz condition. Here, $f:\R \to \mathbb R$ is continuous function that satisfies $(f_1)$ and the following conditions: \\

\noindent There are $p\in\left(2,\frac{2N+4}{N}\right)$  and $C>0$ such that 
$$
|f(s)|\leq C (1+ |s|^{p-1}) \quad \mbox{for all} \quad  s \in \mathbb R; \leqno{(f_7)}
$$

\noindent For each $s_o>0$ there is a $\sigma_o>0$ such that 
$$
\mathcal{F}(s)=sf(s)-2 F(s)\geq \sigma s^2,\mbox { for all  } |s|\geq s_o. \leqno{(f_8)}
$$

Here, it is important to recall that $(f_8)$ is weaker than $(f_3)$.

\vspace{0.5 cm}

 Related to the $V$, we assume that it is continuous and belongs to one of the following classes:  \\
\noindent {\bf Class 1:} \, $V$ is a periodic function
\begin{itemize}
	\item[$(V_1)$] \,  $V$ is $\mathbb{Z}^N$ periodic continuous function, that is, 
	$$
	V(x+y)=V(x), \quad \forall x \in \mathbb{R}^N \quad \mbox{and} \quad \forall y \in \mathbb{Z}^N.
	$$
	\item[$(V_2)$] \, $\displaystyle \inf_{x \in \mathbb{R}^N}V(x)=V_o>0$.
\end{itemize}  

\noindent {\bf Class 2:} \, $V$ is coercive : 
\begin{itemize}
	\item[$(V_3)$]   
	$$V(x) \to +\infty \quad \mbox{as} \quad |x| \to +\infty. $$
\end{itemize}  

\noindent {\bf Class 3:} \, $V$ is a Barstch \& Wang type potential, that is,
\begin{itemize}
	\item[$(V_4)$]  \, $V$ is a continuous function of the form 
	$$
	V(x)=1+\lambda W(x), 
	$$
	where $\lambda >0$ and $W$ is a nonnegative function. 
\end{itemize}  

Our first result is related to the case where the Ambrosetti-Rabinowiz condition only holds at infinity.

\begin{theorem} \label{T0} Assume $(f_1)-(f_3)$ and $(f_5)'$. Suppose that $V$ belongs to Class 1,2 or 3, then, $(P_1)$ has a ground state solution. When $V$ belongs to Class 3, the existence of solution holds for large $\lambda$. 
\end{theorem}

The Theorem \ref{T0} is a crucial step to prove our second result that establishes the existence of solution for $(P_1)$ without Ambrosetti-Rabinowiz condition and it has the following statement

\begin{theorem} \label{T1} Assume $(f_1),(f_7)-(f_8)$ and that $V$ belongs to Class 1,2 or 3. Then, $(P_1)$ has a ground state solution. When $V$ belongs to Class 3, the existence of solution holds for large $\lambda$. 
	
\end{theorem}

As a consequence of Theorem \ref{T1}, we can consider a more general class of problems like 
$$
\left\{
\begin{array}{l}
-\Delta{u}+V(x)u=f(u) + \beta h(u) , \quad \mbox{in} \quad \mathbb{R}^N,\\
u \in H^{1}(\mathbb{R}^N),
\end{array}
\right.
\leqno{(P_2)}
$$
where $\beta >0$ and $V,f$ satisfy the conditions of Theorem \ref{T1}. Related to the $h$, we assume that it is continuous and satisfies: 
\begin{itemize}
	\item[$(h_1)$] $ \displaystyle{\lim_{s\to 0}\frac{h(s)}{s}=0}$ ;
	\item[$(h_2)$] There are $C>0$ and $r \in (2,2^*)$  such that $|h(s)|\leq C(1+ |s|^{r-1})$ for all $s\in \mathbb R$;
	\item[$(h_3)$] For each $s_o>0$, there is a $\sigma_o>0$ such that $$\mathcal{H}(s)=sh(s)-2 H(s)\geq \sigma_o s^2,\mbox { for all  } |s|\geq s_o.$$
	\item[$(h_4)$] There is $\Upsilon>0$ such that
	$$
0<	\frac{h(\Upsilon)}{f(\Upsilon)}=\max\left\{\frac{h(t)}{f(t)}\,:\, t \in [0,\Upsilon]\right\} \quad \mbox{and} \quad 	\frac{h(\Upsilon)}{f(\Upsilon)}=\min\left\{\frac{h(t)}{f(t)}\,:\, t \in [\Upsilon,+\infty)\right\}.
	$$
\end{itemize}

We would like point out that $(h_4)$ holds, for example, if $h(s)/f(s)$ is a nondecreasing function for $s>0$.

Our third theorem is the following 
\begin{theorem} \label{T2} Assume $(f_1),(f_7)-(f_8)$, $(h_1)-(h_4)$ and that $V$ belongs to Class 1,2 or 3. Then, setting  $\beta^*=\frac{f(\Upsilon)}{h(\Upsilon)}$, there is $\Upsilon^*>0$ such that $(P_2)$ has a nontrivial solution for all $\beta  \in [0, \beta^*)$ and $\Upsilon \in [\Upsilon^*, +\infty)$. As in Theorem 1, the existence of solution for Class 3 holds when $\lambda$ is large.
\end{theorem}
	
The plan of the paper is as follows: Section 2 deals with Theorem \ref{T0}, while Sections 3 and 4 are devoted to the proofs of Theorems \ref{T1} and \ref{T2} respectively. In Section 5, we point out that if $\displaystyle{\frac{h(s)}{f(s)} \to +\infty}$, we also have an existence result replacing  $(h_3)$ by a weaker assumption, namely: \, For each $s_o>0$, there is a $\sigma_o>0$ such that 
$$
\mathcal{H}(s)=sh(s)-2 H(s)\geq -\sigma_o s^2,\mbox { for all  } |s|\geq s_o. 	\leqno{(h'_3)}
$$

\section{Proof of Theorem \ref{T0}}

In this section, we will show the existence of solution for a class of auxiliary problem, which proves Theorem \ref{T0} and  will be used in the proof of Theorem  \ref{T1}. Specifically, we will study the existence of solution for the following class of problem
$$
\left\{
\begin{array}{l}
-\Delta{u}+V(x)u=g(u), \quad \mbox{in} \quad \mathbb{R}^N, \\
u \in H^{1}(\mathbb{R}^N),
\end{array}
\right.
\leqno{(AP)}
$$
where $V$ belongs to Class 1,2 or 3, and $g:\mathbb{R} \to \mathbb{R}$ is a continuous function verifying:

\begin{itemize}
	\item[$(g_1)$] $ \displaystyle{\lim_{s\to 0}\frac{g(s)}{s}=0}$ ;
	\item[$(g_2)$] There are $C>0$ and $p \in\left(2,\frac{2N}{N-2}\right)$  such that $|g(s)|\leq C (1+ |s|^{p-1})$ for all $s\in \mathbb R$ ;
	\item[$(g_3)$] There are $\theta>2$ and $M>0$ such that $sg(s)-\theta G(s)>0$ for all $|s|\geq M$ ;
	\item[$(g_4)$]  $\mathcal{G}(s)=sg(s)-2 G(s)>0$ for all $0<|s|\leq M$.
\end{itemize}

\vspace{0.5 cm}
Observe that $(f_1) -(f_3)$ and $(f_5)'$ imply $(g_1)-(g_4)$.  

\vskip .5 cm
In the sequel, we mention some facts that involve conditions $(g_1)-(g_4)$. It is easy to check that $(g_1)-(g_2)$ imply in the inequality below
\begin{equation}\label{eq00}
|G(s)|\leq c |s|+ \frac cp|s|^{p}, \quad \forall s\in \mathbb R.
\end{equation}

The condition $(g_3)$ is  the usual Ambrosetti-Rabinowitz at infinity, however we would like point out that we are not assuming this condition near the origin, which is usually assumed in whole $\mathbb{R}$, when we are considering elliptic problem in whole $\mathbb{R}^N$. Moreover, we do not assume any monotonicity condition on $\displaystyle{\frac {g(s)}s} $, which is in general assumed in a lot of papers without the Ambrosetti-Rabinowitz condition. Finally, we recall that $(g_3)$ is satisfied if 
\begin{equation}\label{eq000}
\lim_{|s|\to +\infty}\frac{g(s)}{|s|^{p-2}s}=1, \quad \mbox{for some} \quad p>2.
\end{equation}	

From  $(g_1)$, there exists $s_o>0$ such that 
\begin{equation}\label{eq01}
|g(s)s - \theta G(s)|\leq \frac {(\theta-2)V_o |s|^2}{4\theta}, \quad \mbox { for all }|s|\leq s_o.
\end{equation}

After these remarks, we are able to prove the existence of solution for $(AP)$.

\subsection{The variational approach} In this subsection, related to the function $V$, we assume only condition $(V_2)$. From now on, we set   
$$E=\left\{u\in D^{1,2}(\mathbb{R}^N);\,\,\int_{\mathbb{R}^N}(|\nabla u|^2+V(x)|u|^2)\,dx<+\infty\right\}$$
endowed with the norm
$$\|u\|=\left(\int_{\mathbb{R}^N}(|\nabla u|^2+V(x)|u|^2)\,dx\right)^{\frac{1}{2}}$$
and functional $J:E\to\R$ by
$$
J(u)=\frac{1}{2}\|u\|^2 \, dx-\int_{\mathbb{R}^N} G(u)\,dx.
$$
From $(g_1)-(g_3)$,  $J$ is well defined in $E$ and $J \in C^{1}(E, \mathbb{R})$ with 
$$
J'(u)v=\int_{\mathbb{R}^N}(\nabla u \nabla v + uv)\,dx - \int_{\mathbb{R}^N}g(u)v\,dx, \quad \forall u,v \in E.
$$
Moreover, it is very easy to check that $J$ also satisfies the mountain pass geometry. In what follows, we denote by $c>0$ the mountain pass level associated with $J$, that is,  
\begin{equation} \label{MPL}
c=\inf_{\gamma \in \Gamma}J(\gamma(t))>0 
\end{equation}
where
$$
\Gamma =\{\gamma\in C([0,1],E):\gamma(0)=0\mbox{ and } J(\gamma(1))< 0\}.
$$ 
Associated with $c$, we have a Cerami sequence $(v_n) \subset E$, that is, 
\begin{equation} \label{C1}
J(v_n) \to c \quad \mbox{and} \quad (1+\|v_n\|)\|J'(v_n)\| \to 0 \quad \mbox{as} \quad n \to +\infty. \,\,\, ( \mbox{See \cite[Theorem 5.46]{MMP}} )
\end{equation}

\begin{proposition} The sequence $(v_n)$ is bounded in $E$.
	
\end{proposition}	
\begin{proof} As $(v_n)$ is a Cerami sequence,  we have  
$$
J(v_n)=c+o_n(1) \quad \mbox{and} \quad J'(v_n)v_n=o_n(1). 
$$
For each $n$, we set 
\[
\mathcal{A}_n=\{x\in \mathbb R^N: s_o\leq|v_n(x)|\leq M\},
\]
where $s_o$ was fixed in (\ref{eq01}). From $(g_4)$, 
\[
\sigma_o|\mathcal{A}_n|\leq \int_{\mathcal{A}_n}\mathcal{G}(v_n)\,dx\leq J(v_n)-\frac 12 J'(v_n)v_n=c+o_n(1),
\]
where $\sigma_o=\inf\{\mathcal{G}(s): |s|\in[s_o,M]\}>0$ and $|\mathcal{A}_n|$ denotes the Lebesgue's measure of $\mathcal{A}_n$. This inequality implies that $(|\mathcal{A}_n|)$ is a bounded sequence. In the sequel, we consider the sets
$$
\mathcal{C}_n=\{x\in \mathbb R^N: |v_n(x)|< s_o\}
$$ 
and 
$$
\mathcal{B}_n=\{x\in \mathbb R^N: |v_n(x)|> M\}.
$$
From (\ref{eq01}),
\begin{eqnarray*}
	\int_{\mathcal{C}_n}[v_ng(v_n)-\theta G(v_n)]dx\geq -\frac {(\theta-2)V_o }{4}\int_{\mathcal{C}_n} |v_n|^2dx\geq-\frac {(\theta-2) }{4}\|v_n\|^2, \quad \forall n \in \mathbb{N}
\end{eqnarray*}
and by $(g_3)$, 
\[
\int_{\mathcal{B}_n}[v_ng(v_n)-\theta G(v_n)]dx\geq 0, \quad \forall n \in \mathbb{N}.
\]
On the other hand, $(g_3)$ gives
\begin{eqnarray*}
	\int_{\mathcal{A}_n}[v_ng(v_n)-\theta G(v_n)]dx\geq -|\mathcal{A}_n| M_o, \quad \forall n \in \mathbb{N},
\end{eqnarray*}
for some $M_0>0$. The above analysis ensure that  
$$
\begin{array}{l}
	c+o_n(1)= \displaystyle J(v_n)-\frac 1 \theta J'(v_n)v_n=\frac {(\theta-2) }{2\theta}\|v_n\|^2 +\frac 1\theta \int_{\mathbb R^N}[v_ng(v_n)-\theta G(v_n)]dx \\
	\mbox{} \\
 \;\;\;\;\;\;\;\;\;\;\;\;\;\;\; \geq \displaystyle \frac {(\theta-2) }{4\theta}\|v_n\|^2-M_o|\mathcal{A}_n|, \quad \forall n \in \mathbb{N}.
\end{array}
$$
Since $(|\mathcal{A}_n|)$ is a bounded sequence, it follows that $(v_n)$ is bounded in $E$.

\end{proof}

\subsection{Existence of solution for $(AP)$: The periodic case}

In this section we assume that $V$ verifies $(V_1)-(V_2)$. By Subsection 2.1, we know that there is a bounded Cerami sequence $(v_n)$ associated with the mountain pass level $c$, that is,
$$
J(v_n) \to c \quad \mbox{and} \quad (1+\|v_n\|)\|J'(v_n)\| \to 0 \quad \mbox{as} \quad n \to +\infty. 
$$
From boundedness of $(v_n)$, we deduce that $(v_n)$ is a $(PS)_c$ sequence, that is 
$$
J(v_n) \to c \quad \mbox{and} \quad J'(v_n) \to 0 \quad \mbox{as} \quad n \to +\infty. 
$$
Moreover, we can assume that for some subsequence, there is $u \in E$ such that
$$
v_n \rightharpoonup v \quad \mbox{in} \quad E,
$$
$$
v_n \to v \quad \mbox{in} \quad L^{t}(B_R(0)), \quad  \forall t \in [1,2^{*}) \quad \mbox{and} \quad \forall R>0 ,
$$
and
$$
v_n(x) \to v(x) \quad \mbox{a.e in} \quad \mathbb{R}^N.
$$

Since $J$ is invariant by translation,  by Lions \cite{Lions}, we can assume that $v \not= 0$. Hence, $v$ is a nontrivial critical point for $J$, and so, $v$ is a nontrivial solution for $(AP)$.  The reader can see more details of how we can use \cite{Lions} in \cite{ACM}, \cite{ZElati}, \cite{ZR} and \cite{Szulkin1}.

\subsection{The existence of solution for $(AP)$ : The Coercive case}

In this case, it is well known that the following compact embedding holds
\begin{equation} \label{compacidade}
E \hookrightarrow L^{t}(\mathbb{R}^N), \quad \forall t \in (2,2^*) \,\, (\,\, \mbox{See \cite{Costa2}} \,\, ).
\end{equation}

This compact embedding together with the boundedness of $(PS)$ sequences permit to prove that $J$ verifies the well known (PS) condition, and so, the mountain pass level $c$ is a positive critical value for $(AP)$. This prove that $(AP)$ has a nontrivial solution. 

We would like to point out that the same conclusion holds if  $(V_3)$ is replaced by  
$$
\mbox{For all} \quad K>0 \quad \mbox{we have that} \quad |\left\{x \in \mathbb{R}^N\,:\, V(x) \leq K \right\}|<\infty, \,\, (\,\, \mbox{See \cite{Bartsch-Wang1} } \,\, ) 	\leqno{(V_4)} 
$$
where $|A|$ denotes the Lebesgue's measure of a mensurable set $A \subset \mathbb{R}^N$. The last condition also implies in the compact embedding (\ref{compacidade}). 

Finally, we would like point out that if $V$ is radially symmetric, the compactness (\ref{compacidade}) also holds in $H_{rad}^{1}(\mathbb{R}^N)$, for more details see Willem \cite[Corollary 1.26]{Willem}. The existence of solution follows as in \cite[Section 1.6]{Willem}, where the above compact embedding (\ref{compacidade}) and the Principle of symmetric criticality due to Palais apply an important role in the arguments.

\subsection{Existence of solution for (AP): $V$ is a Barstch \& Wang type potential} 
The approach explored in the previous subsection can be also used to study the existence of solution when the potential $V$ is of the form
$$
V(x)=1+\lambda W(x), \quad \forall x \in \mathbb{R}^N,  \,\, (\,\, \mbox{See \cite{BW2} } \,\, ) 
$$ 
where $\lambda >0$ and $W:\mathbb{R}^N \to \mathbb{R}$ is a nonnegative continuous function that satisfies 
$$
\mbox{There is} \quad K>0 \quad \mbox{such that} \quad |\left\{x \in \mathbb{R}^N\,:\, V(x) \leq K \right\}|<+\infty. 	\leqno{(W_1)}
$$

Arguing as \cite{BW2}, there is $\alpha>0$, which is independent of $\lambda$, such that 
\begin{equation} \label{EBW1}
\int_{\mathbb{R}^N}|v_n|^{p}\,dx \geq \alpha, \quad \forall n \in \mathbb{N},  \,\, (\,\, \mbox{See \cite[Lemma 2.4]{BW2} } \,\, ) .
\end{equation}
Moreover, it is possible to prove that there are $R>0$ independent of $\lambda>0$, and $\lambda^*>0$ such that 
\begin{equation} \label{EBW2}
\limsup_{n \to +\infty}\int_{B^{c}_R(0)}|v_n|^{p}\,dx \leq \frac{\alpha}{2}, \quad \forall n \in \mathbb{N} \quad \mbox{and} \quad \lambda \geq \lambda^*. \,\, (\,\, \mbox{See \cite[Lemma 2.5]{BW2} } \,\, ) 
\end{equation}
From (\ref{EBW1})-(\ref{EBW2}),
$$
\int_{B_R(0)}|v|^{p}\,dx \geq \frac{\alpha}{2}>0, \quad \forall \lambda \geq \lambda^*,
$$
showing that $v$ is nontrivial. Hence $(AP)$ has a nontrivial solution for $\lambda$ large enough.

\vskip .5 cm
Before concluding this section, the reader is invited to see that the results showed in this section prove the Theorem \ref{T0}.

\section{Proof of Theorem \ref{T1}}

In this section we deal with the proof of Theorem \ref{T1}, and the results obtained in Section 2 will be crucial in our approach. 

In what follows, we consider the auxiliary problem
$$
\left\{
\begin{array}{l}
-\Delta u+ V(x)u= f(u) +\varepsilon|u|^{q-2}u \,\,\,\,\mbox{in }\,\,\,\,\mathbb{R}^N, \\
u \in E.
\end{array}
\right.
\leqno{(AP)_\varepsilon}
$$
where $\displaystyle{q=\frac {2N+4}N}$  and $\varepsilon >0$.

It is easy to check that $g_\varepsilon(s)=f(s) +\varepsilon|s|^{q-2}u$ satisfies $(g_1)-(g_4)$ for $\theta=q$. Hence, from the previous sections, for fixed $\varepsilon>0$, there is a nontrivial solution $u_\varepsilon$ of $(AP)_\varepsilon$ such that 
$$
J_\varepsilon(u_\varepsilon)=c_\varepsilon \quad \mbox{and} \quad J'_\varepsilon(u_\varepsilon)=0 
$$
where 
$$
J_\varepsilon(u)=\frac{1}{2}\|u\|^2 \, dx-\int_{\mathbb{R}^N} G_\varepsilon(u)\,dx, \quad \forall u \in E, 
$$
and $c_\varepsilon$ is the mountain pass level associated with $J_\varepsilon$, that is,
$$
0<c_\varepsilon=\inf_{\gamma \in \Gamma}J_\varepsilon(\gamma(t)).
$$
Since $J_\varepsilon(v)\leq J_0(v)$ for all $v\in E$, we have $c_\varepsilon\leq c_0$ for all $ \varepsilon >0$. Furthermore, it is possible to prove that there is $\alpha>0$, which is independent of $\varepsilon \in [0,1]$, such that  
\begin{equation} \label{limitab}
c_\varepsilon \geq \alpha, \quad \forall \varepsilon \in [0,1].
\end{equation}
As an immediate consequence of the above inequality, we have 
\begin{equation} \label{limitab*}
\|u_\varepsilon \| \not\to 0 \quad \mbox{as} \quad \varepsilon \to 0.   
\end{equation}

The next result establishes an important estimate for $\{u_\varepsilon\}$ in $E$, which is a key point in our approach.   

\begin{proposition} \label{LIMITACAO} There is $\varepsilon_0>0$ such that $\{u_\varepsilon\}$ is bounded in $E$  for all $\varepsilon \in [0,\varepsilon_0]$.
	
\end{proposition}

\begin{proof}
For each $\varepsilon >0$, we have $\mathcal{G}_\varepsilon \geq \mathcal{F}$. Thus, by $(f_8)$,
$$
\sigma\int_{\{|u_\varepsilon|\geq s_o\}}|u_\varepsilon|^2dx\leq \int_{\mathbb R^N}\mathcal{G}_\varepsilon(u_\varepsilon)\,dx= 2J_\epsilon(u_\varepsilon)- J_\varepsilon'(u_\varepsilon)u_\varepsilon= 2c_\varepsilon \leq 2c_0.
$$
Arguing as (\ref{eq01}), there exists $s_o>0$, which is independent of $\varepsilon >0$ small, such that 
\begin{equation}\label{eq02}
G_\varepsilon(s)\leq \frac {V_o |s|^2}{4}\mbox { for all }|s|\leq s_o.
\end{equation}
On the other hand, from $(f_7)$,  for each $\varepsilon>0$, there is $C_\varepsilon >0$ verifying  
\begin{equation}\label{eq03}
G_\varepsilon(s)\leq \varepsilon |s|^q+2C_\varepsilon |s|^2\mbox { for all }|s|\geq s_o.
\end{equation}
So, setting 
$$
\mathcal{C}_\varepsilon=\{x\in \mathbb R^N: |u_\varepsilon(x)|< s_o\},
$$ 
and using  (\ref{eq02}) and (\ref{eq03}), we get
$$
\begin{array}{l}
	\frac 12\|u_\varepsilon\|^2=c_\varepsilon+ \displaystyle \int_{\mathbb R^N}G_\varepsilon (u_\varepsilon)dx \leq c_\varepsilon+\int_{\mathcal{C}_\varepsilon}G_\varepsilon(u_\varepsilon)dx +\varepsilon\int_{\{|u_\varepsilon|\geq s_o\}}|u_\varepsilon|^qdx+ 2C_\varepsilon\int_{\{|u_\varepsilon|\geq s_o\}}|u_\varepsilon|^2dx\\ 
	\mbox{}\\
\hspace{5 cm}	\leq c_0+\frac{C_\varepsilon c_0}{\sigma}+\frac 14 \displaystyle \int_{\mathcal{C}_\varepsilon}V(x)|u_\varepsilon|^2dx +\varepsilon\int_{\{|u_\varepsilon|\geq s_o\}}|u_\varepsilon|^qdx,
\end{array}
$$
that is,
\begin{eqnarray*}
	\frac 14\|u_\varepsilon\|^2\leq \frac{(C_\varepsilon+\sigma)c_0}{\sigma}+\varepsilon \int_{\{|u_\varepsilon|\geq s_o\}}|u_\varepsilon|^qdx.
\end{eqnarray*}
Recalling that
\begin{eqnarray*}
	\int_{\{|u_\varepsilon|\geq s_o\}}|u_\varepsilon|^qdx\leq \left( \int_{\{|u_\varepsilon|\geq s_o\}}|u_\varepsilon|^{2^*}dx\right)^\frac{N-2}{N}\left( \int_{\{|u_\varepsilon|\geq s_o\}}|u_\varepsilon|^{\frac{N(q-2)}{2}}dx\right)^\frac{2}{N}\\\leq \left( \frac{c_0}{\sigma}\right)^\frac{2}{N}\|u_\varepsilon\|_{2^*}^2\leq \left( \frac{c_0}{\sigma}\right)^\frac{2}{N}S^{-1}\|u_\varepsilon\|^2,
\end{eqnarray*}
and fixing $\varepsilon_0=\frac{S}{8}(\frac{\sigma}{c_0})^{\frac{2}{N}}$, we derive that 
\[
\frac 18\|u_\varepsilon\|^2\leq \frac{(C_\varepsilon +\sigma)c_0}{\sigma}, \quad \forall \varepsilon \in (0,\varepsilon_0].
\]
This proves the boundedness of $\{u_\varepsilon\}$ for $\varepsilon \in (0,\varepsilon_0]$ in $E$. 

\end{proof}

In the sequel, for each $n \in \mathbb{N}$, we denote by $u_n$ the solution of $(AP)_{\frac{1}{n}}$, that is, 
$$
-\Delta u_n+ V(x)u_n= f(u_n) +\frac{1}{n}|u_n|^{q-2}u_n \,\,\,\,\mbox{in }\,\,\,\,\mathbb{R}^N. \leqno{(AP)_{\frac{1}{n}}}
$$

By Proposition \ref{LIMITACAO}, the sequence $(u_n)$ is bounded in $E$, hence for some subsequence, there is $u \in E$ such that
$$
u_n \rightharpoonup u \quad \mbox{in} \quad E,
$$
$$
u_n \to u \quad \mbox{in} \quad L^{t}(B_R(0)), \quad  \forall t \in [1,2^{*}) \quad \mbox{and} \quad \forall R>0, 
$$
and
$$
u_n(x) \to u(x) \quad \mbox{a.e in} \quad \mathbb{R}^N.
$$

From this, for each $\phi \in E$, 
$$
\int_{\mathbb{R}^N}\nabla u_n \nabla v \,dx + \int_{\mathbb{R}^N}V(x)u_n\phi\, dx= \int_{\mathbb{R}^N}f(u_n)\phi \, dx+ \frac{1}{n}\int_{\mathbb{R}^N}|u_n|^{q-2}u_n\phi\,dx, \quad \forall n \in \mathbb{N}. 
$$
Then taking the limit of $n \to +\infty$, we find
$$
\int_{\mathbb{R}^N}\nabla u \nabla v \,dx + \int_{\mathbb{R}^N}V(x)u \phi\, dx= \int_{\mathbb{R}^N}f(u)\phi\,dx, \quad \forall  \phi \in E,
$$
showing that $u$ is a solution of $(P_1)$. Our goal is proving that $u$ is a nontrivial solution.  Have this in mind, firstly we prove the result below

\begin{lemma} \label{naopode} $ u_n \not\to 0$ in $L^{p}(\mathbb{R}^N)$, where $p$ was given in $(f_7)$. 
	
\end{lemma}
\begin{proof}  Since $u_n$ is a solution of $(AP)_{\frac{1}{n}}$, it follows that
$$
\|u_n\|^{2}=\int_{\mathbb{R}^N}f(u_n)u_n\,dx+\frac{1}{n}\int_{\mathbb{R}^N}|u_n|^{q}\,dx, \quad \forall n \in \mathbb{N}.
$$
By $(f_1)$ and $(f_7)$, 
$$
|f(t)t| \leq \frac{V_o}{2}|t|^{2}+C|t|^{p}, \quad \forall t \in \mathbb{R},
$$
for some positive constant $C$. Consequently, 
$$
\|u_n\|^{2} \leq C \int_{\mathbb{R}^N}|u_n|^{p}\,dx+\frac{1}{n}\int_{\mathbb{R}^N}|u_n|^{q}\,dx, \quad \forall n \in \mathbb{N},
$$
for some positive constant $C$. Using the fact that $(u_n)$ is bounded in $E$, we have that $(u_n)$ is also bounded in $L^{q}(\mathbb{R}^N)$. Thus, supposing by contradiction that $u_n  
\to 0$ in $L^{p}(\mathbb{R}^N)$, we obtain
$$
u_n \to 0 \quad \mbox{in} \quad E,
$$
which contradicts (\ref{limitab*}).
\end{proof}

\noindent {\bf Conclusion of the proof of Theorem \ref{T1}} \\
Now, we can argue as in Subsections 2.2, 2.3 and 2.4 to deduce that $u \not= 0$. For example, in the periodic case, the Lions' result together with Lemma \ref{naopode} ensures that we can assume that $u \not= 0$. For the others cases, we argue exactly as in Section 2.   

\section{Proof of Theorem \ref{T2}}

In what follows, we consider the function $h_\Upsilon:\mathbb{R} \to \mathbb{R}$ given by
$$
h_\Upsilon(t)=
\left\{
\begin{array}{l}
h(t), \quad t \in [0,\Upsilon], \\
\mbox{}\\
\frac{h(\Upsilon)}{f(\Upsilon)}f(t), \quad t \geq \Upsilon,
\end{array}
\right.
$$ 
and $f_\Upsilon:\mathbb{R} \to \mathbb{R}$ by
$$
f_\Upsilon(t)=
\left\{
\begin{array}{l}
0, \quad t \in (-\infty,0], \\
f(t)+\beta h_\Upsilon(t), \quad t \geq 0.
\end{array}
\right.
$$ 
Using the above notations, our intention is proving the existence of a nontrivial solution $u_{\beta}$ with $\|u_{\beta}\|_{\infty}<\Upsilon$ for the following auxiliary problem  
$$
\left\{
\begin{array}{l}
-\Delta{u}+V(x)u=f_\Upsilon(u), \quad \mbox{in} \quad \mathbb{R}^N, \\
u \in E.
\end{array}
\right.
\leqno{(AP)_{\beta,\Upsilon}}
$$
Hereafter, we fix $\beta^*=\frac{f(\Upsilon)}{h(\Upsilon)}$ and $\beta \in [0, \beta^*]$. A simple computation shows that $f_\Upsilon$ satisfies $(f_1),(f_7)-(f_8)$, more precisely,

\begin{itemize}
	\item[$(i)$] $ \displaystyle{\lim_{s\to 0}\frac{f_\Upsilon(s)}{s}=0}$ ;
	\item[$(ii)$] There are $C>0$ and $p\in\left(2,\frac{2N+4}{N}\right)$  such that $|f_\Upsilon(s)|\leq C (1+ |s|^{p-1})$ for all $s\in \mathbb R$;
	\item[$(iii)$] For each $s_o>0$, there is $\sigma>0$ such that $$\mathcal{F}_\Upsilon(s)=sf_\Upsilon(s)-2 F_\Upsilon(s)\geq \sigma s^2,\mbox { for all  } s\geq s_o,$$
	where $F_\Upsilon(s)=\int_{0}^{s}f_\Upsilon(t)\,dt$.
\end{itemize}
Indeed, we will check only $(ii)-(iii)$, because $(i)$ is immediate. By $(h_4)$, 
$$
h_\Upsilon(s) \leq \frac{h(\Upsilon)}{f(\Upsilon)}f(s), \quad \forall s \geq 0,
$$
leading to
\begin{equation} \label{NOVA}
H_\Upsilon(s)=\int_{0}^{s}h_\Upsilon(t)\,dt \leq \frac{h(\Upsilon)}{f(\Upsilon)}F(s), \quad \forall s \geq 0,
\end{equation}
and 
$$
|f_\Upsilon(s)| \leq 2|f(s)|\leq C(1+|s|^{p-1}), \quad \mbox{for} \quad p\in\left(2,\frac{2N+4}{N}\right).
$$
On the other hand, if $s \in [0,\Upsilon]$, we have 
\begin{equation} \label{NOVA1}
\mathcal{F}_\Upsilon(s)= sf(s)-2 F(s)+\beta(sh(s)-2 H(s)). 
\end{equation}
Now, if $s \geq \Upsilon$,
$$
\mathcal{F}_\Upsilon(s)=sf(s)-2 F(s)+\beta\frac{h(\Upsilon)}{f(\Upsilon)}(sf(s)-2F(s))+H(\Upsilon)-\frac{h(\Upsilon)}{f(\Upsilon)}F(\Upsilon) 
$$
and so, by (\ref{NOVA}),
$$
\mathcal{F}_\Upsilon(s) \geq \left(1+ \beta\frac{h(\Upsilon)}{f(\Upsilon)}\right)(sf(s)-2 F(s)) .
$$
As a consequence of the above analysis,  
$$
\mathcal{F}_\Upsilon(s) \geq \sigma s^{2} \quad \mbox{for} \quad s \geq s_o.
$$

From this, we can apply Theorem \ref{T1} to get a solution $u_\beta \in E$ of $(AP)_{\beta,\Upsilon}$. Our next step is proving that there is $\Upsilon^*>0$ such that 
$ \|u_\beta\|_\infty \leq \Upsilon^*$ for $\Upsilon \geq \Upsilon^*$, because this estimate permits to conclude that $u_\beta$ is a solution of $(P_2)$ for $\Upsilon \geq \Upsilon^*$, which shows Theorem \ref{T2}. However, the existence of $\Upsilon^*$ follows from the following fact: If $I_\Upsilon:E \to \mathbb{R}$ denotes the energy functional associated with $(AP)_{\beta,\Upsilon}$, that is,
$$
I_\Upsilon(u)=\frac{1}{2}\|u\|^2 \, dx-\int_{\mathbb{R}^N} F_\Upsilon(u)\,dx, \quad \forall u \in E,
$$
we have
$$
I_\Upsilon(u_\beta) \leq J_0(u), \quad \forall u \in E \quad \mbox{and} \quad \beta \geq 0.
$$
Hence,
$$
I_\Upsilon(u_\beta) \leq c_0, \quad \forall \Upsilon \geq 0,
$$
where $c_0$ denotes the mountain pass level associated with the functional $J_0$ ( see Section 3). Since $I'_\Upsilon(u_\beta)=0$, we can argue as in Section  3 to show that there is $k>0$, which is independent of $\Upsilon$ and $\beta$, such that
$$
\|u_\beta\| \leq k, \quad \forall \Upsilon >0 \quad \mbox{and} \quad \forall \beta \in [0,\beta^*].
$$
Recalling that by $(ii)$, 
$$
|f_\Upsilon(s)|\leq C (1+ |s|^{p-1}) \quad \forall s\in \mathbb R,
$$
where the constant $C$ does not depend on $\beta$ and $\Upsilon$, the bootstrap argument found \cite[Proposition 2.15]{Rab} ensures that there is $K_*>0$, which is independent of $\beta$ and $\Upsilon$ such that
$$
\|u_\beta\|_\infty \leq K, \quad \forall \Upsilon >0 \quad \mbox{and} \quad  \beta \in [0,\beta^*].
$$ 
From this, if $\Upsilon \geq \Upsilon^*=K$, it follows that $u_\beta$ is a solution for $(P_2)$, finishing the proof of Theorem \ref{T2}.

\section{Final comments}
In this section we would like point out that if $\displaystyle{\frac{h(s)}{f(s)} \to +\infty}$ as $s \to +\infty$, we can replace $(h_3)$ by 
\begin{itemize}
	\item[$(h'_3)$] For each $s_o>0$, there is a $\sigma_o>0$ such that $$\mathcal{H}(s)=sh(s)-2 H(s)\geq -\sigma_o s^2,\mbox { for all  } |s|\geq s_o.$$
\end{itemize}
Indeed, using this assumption in (\ref{NOVA1}), we get
$$
\mathcal{F}_\Upsilon(s) \geq (\sigma -\beta \sigma_o)s^2, \quad \forall s \in [0,\Upsilon].
$$
If $\displaystyle{\frac{h(s)}{f(s)} \to +\infty}$ as $s \to +\infty$, we derive that $\beta^* \to 0$ as $\Upsilon^* \to +\infty$. Hence,  
$$
(\sigma -\beta \sigma_o)>0, \quad \forall \beta \in (0, \beta^*),
$$
for $\Upsilon^*$ large enough. From this, Theorem \ref{T2} still holds for $\Upsilon^*$ large enough.

\end{document}